\documentclass[12pt,a4paper,twoside]{article}
\usepackage[UKenglish]{babel}
\usepackage[T1]{fontenc}
\usepackage[utf8x]{inputenc}
\usepackage{latexsym,amsfonts,amsmath,amsthm,amssymb,mathrsfs}
\usepackage{fullpage}
\usepackage{xcolor,bm, bbm}
\usepackage{paralist}
\usepackage{todonotes}
\usepackage{dsfont}
\usepackage{float}
\usepackage{caption}
\usepackage{xfrac}
\usepackage[colorlinks, linkcolor=blue,citecolor=blue, anchorcolor=blue,urlcolor=blue,hypertexnames=false]{hyperref}

\usepackage{graphicx}
\usepackage{tikz,ifthen}
\usepackage{pgfplots}

\DeclareMathOperator{\N}{\mathbb{N}}

\def\d{ \,\mathrm{d}}
\def\x{{\bf{x}}}
 \def\n{{\bf{n}}}
 \def\nx{{\bf{n}\cdot{\bf{x}}}}
\def\k{{\bf{k}}}

\newtheorem{thm}{Theorem}

\newtheorem{lem}{Lemma}[section]

\newtheorem{cor}[thm]{Corollary}

\newtheorem{proposition}[lem]{Proposition}

\newtheorem*{rem*}{Remark}
\newtheorem*{thm*}{Theorem}

\title{\bf Non-Salem sets in multiplicative Diophantine approximation}
\author{Bo Tan and Qing-Long Zhou\footnote{Corresponding author.}
}

\medskip

\date{}

\begin{document}


\maketitle

\begin{abstract}
In this paper, we answer a question of Cai-Hambrook in (arXiv$\colon$ 2403.19410).
Furthermore, we compute the   Fourier dimension of the multiplicative $\psi$-well approximable set
$$M_2^{\times}(\psi)=\left\{(x_1,x_2)\in [0,1]^{2}\colon \|qx_1\|\|qx_2\|<\psi(q) \text{ for infinitely many } q\in \N\right\},$$
where $\psi\colon\N\to [0,\frac{1}{4})$ is a positive function satisfying $\sum_q\psi(q)\log\frac{1}{\psi(q)}<\infty.$
As a  corollary, we show that the set $M_2^{\times}(q\mapsto q^{-\tau})$ is non-Salem for $\tau>1.$
\end{abstract}

\section{Introduction}
\subsection{Fourier dimension, Random Salem and non-{Salem} sets}
The regularity properties of a function/measure and the decay rate of its Fourier transform are tightly related.
The study of the optimal Fourier decay rate of measures supported on a fractal set $E\subseteq \mathbb{R}^{n}$ is
a central problem in analysis exploring the interplay between harmonic analysis and fractal geometry.

The optimal power-like decay of the Fourier transform is used to define the Fourier dimension of a  Borel set $E\subseteq \mathbb{R}^{n}$:
$$\dim_{\rm F}E=\sup\big\{s\in[0,n]\colon \exists \mu\in \mathcal{M}(E) \text{ such that } |\widehat{\mu}(\xi)|\ll_s(1+|\xi|)^{-s/2}\big\}.\footnote{Thoughout we use Vingradov notation: $A\ll B$ means $|A|\le C|B|$ for some constant $C>0$;  $A\asymp B$ means $A\ll B$ and $B\ll A$. }$$
Here  $\mathcal{M}(E)$ denotes the set of Borel probability measure on $\mathbb{R}^{n}$ that give full measure to $E.$
Fourier dimension is closely related to Hausdorff dimension. Indeed, Frostman's lemma \cite{M95,M15} states that the Hausdorff dimension of a Borel set $E$  is equal to
$$\dim_{\rm H}E=\sup\left\{s\in[0,n]\colon \exists \mu\in \mathcal{M}(E) \text { such that }\int |\widehat{\mu}(\xi)|^{2}|\xi|^{s-1}\d\mu<\infty\right\}.$$
Hence we obtain that
$$\dim_{\rm F}E\le \dim_{\rm H}E$$
for every Borel set $E\subseteq \mathbb{R}^{n}.$ In the case when the equality holds for a set $E,$
it is called a \emph{Salem set} or  \emph{round set} \cite{K93}. There is no lack of Salem sets; many random sets are Salem.
For example,  Salem \cite{S51} proved that for every $s\in[0,1]$ there exists a Salem set with dimension $s$ by constructing random Cantor-type sets in $\mathbb R$.
Kahane \cite{K66} showed that for every $s\in[0,n]$ there exists a Salem set in $\mathbb{R}^{n}$ with dimension $s$ by considering images of Brownian motion.
\L aba-Pramanik \cite{LP09} then applied these to the additive structure of Brownian images.
Later, Shieh-Xiao \cite{SX06} extended Kahane’s work to very general classes of Gaussian random fields.
For other random Salem sets the readers are referred to \cite{B96, CS17, E16, LP09, M18,SS18} and references therein.
On the other hand, some naturally defined random sets are not Salem. Fraser-Orponen-Sahlsten \cite{FOS14} showed that
the Fourier dimension of the graph of any function defined on $[0, 1]$ is at most 1, which in turn shows that graph of fractional Brownian motion is not Salem almost surely.
Fraser-Sahlsten \cite{FS18} further showed that the Fourier dimension of the graph of fractional Brownian motion is 1 almost surely.

In this papre, our motivation is to find more explicit Salem or non-Salem sets in the theory of metric Diophantine approximation.

\subsection{Metric Diophantine approximation}
Metric Diophantine approximation is concerned with the quantitative analysis of the density of rationals in the reals.

For an approximation function $\psi\colon \N \to [0,\frac{1}{2})$, the $\psi$-well approximable set $W(\psi)$ is defined to be
$$W(\psi):=\big\{x\in [0,1]\colon |\!|qx|\!|<\psi(q) \text{ for infinitely many } q\in \N \big\},$$
where $|\!|\alpha|\!|:=\min\{|\alpha-m|\colon m\in \mathbb{Z}\}$ denotes the distance from a real number $\alpha$ to the nearest integer.

The classical Khintchine’s theorem \cite{K24} states that, if $\psi$ is non-increasing, the Lebesgue measure $\mathcal{L}(W(\psi))=0$ or 1 according as the series $\sum_q\psi(q)$ converges or diverges.
Duffin-Schaeffer \cite{DS41} proved that Khintchine's theorem generally fails without the monotonicity condition on $\psi$. More precisely, they constructed a function $\psi$ which is supported on a set of very smooth integers (having a large number of small prime factors), such that $\sum_q\psi(q)$ diverges, but  $W(\psi)$ is null.   Further, Duffin-Schaeffer conjectured that for almost all $x\in[0,1]$ there are infinitely many coprime pairs $(p, q)$ such that $|qx-p|<\psi(q)$ if and only if $\sum_q\frac{\phi(q)}{q}\psi(q)$ diverges, where $\phi$ is the Euler's totient function. After important contributions of Gallagher \cite{G61}, Erd\"{o}s \cite{E70}, Vaaler \cite{V78}, Pollington-Vaughan \cite{PV90}, Beresnevich-Velani \cite{BV06}, the Duffin-Schaeffer conjecture was  solved affirmatively by Koukoulopoulos-Maynard \cite{KM20}.

Jarn\'{i}k Theorem \cite{J31} shows, under the monotonicity of $\psi,$  that
$$\dim_{\rm H}W(\psi)=\frac{2}{\tau+1}, \ \ \text{where } \tau=\liminf_{q\to\infty}\frac{-\log \psi(q)}{\log q}.$$
It is worth mentioning that Jarn\'{i}k Theorem can be deduced by Khintchine's Theorem via the mass transference principle of Beresnevich-Velani \cite{BV06}.
For a general function $\psi,$ the Hausdorff dimension of the $\psi$-well approximable set $W(\psi)$ was studied extensively in Hinokuma-Shiga \cite{HS96}.

We now turn to discuss the Fourier dimension of $W(\psi).$
For $\psi(q)=q^{-\tau}$, Kaufman \cite{K81} proved that the set $W(\psi)$ is of Fourier dimension $\frac{2}{\tau+1}$ for $\tau>1;$
this result is expounded in  Bluhm \cite{B98}.
Notably, this is the first explicit non-random construction of a Salem set of   dimension other than 0 or 1 in $\mathbb{R}$.
Moreover, Kaufman's result hinted an approach of find explicit non-random Salem sets in high dimension space 
 \cite{FH23,H17,H19}.
Recently, Cai-Hambrook  generalized Kaufman's result by considering the following set\footnote{Strictly speaking, Cai-Hambrook considered the mutli-dimensional generalization of $W(\psi,Q).$}
$$W(\psi,Q):=\big\{x\in [0,1): |\!|qx|\!|<\psi(q) \text{ for infinitely many } q\in Q \big\}.$$
\begin{thm}[Cai-Hambrook, \cite{CH24}]\label{CH24}
Let $Q$ be an infinite subset of $\mathbb{N}.$ Let $\psi\colon \N\to [0,\frac{1}{2})$ be an arbitrary function satisfying $\sum_{q\in Q}\psi(q)<\infty$. Then
$$\dim_{\rm F}W(\psi,Q)=\min\big\{2\lambda(\psi),1\big\},$$
where $\lambda(\psi)=\inf\left\{s\in[0,1]\colon \sum_{q\in Q}\left(\frac{\psi(q)}{q}\right)^{s}<\infty\right\}.$
\end{thm}

 Furthermore, Cai-Hambrook  proposed the following question.
\smallskip

\noindent\textbf{Prove or disprove}$\colon$ If $\sum_{q\in Q}\psi(q)=\infty,$ then
$$\dim_{\rm F}W(\psi,Q)=1.$$

We provide a negative answer to this question.
\begin{thm}\label{TZ1}
There exists an approximation function $\psi$ satisfying
$\sum_{q\in \N}\psi(q)=\infty$ and
$$\dim_{\rm F}W(\psi,\N)=0.$$
\end{thm}

\subsection{Multiplicative Diophantine approximation}
 The study of Multiplicative Diophantine approximation is motivated by \textbf{Littlewood conjecture} \cite{BRV16}$\colon$ for any pair $(x_1,x_2)\in[0,1]^{2},$
$$\liminf_{q\to\infty}q\|qx_1\|\|qx_2\|=0.$$
Littlewood's conjecture has attracted much attention, see \cite{BV11, CT24, CT24+,CY24,EKL06,PV00,Y23} and references within.
Despite some   remarkable progress,   Littlewood conjecture remains very much open.
Along the way, there have been significant advances towards the corresponding metric theory.
The first systematic result\footnote{The convergence part was already known  (\cite{BV15}, Remark 1.2). For this reason, Gallagher’s theorem sometimes refers to the divergence part alone.} in this direction is a famous theorem of Gallagher \cite{G61}. Given $\psi\colon \N\to [0,\frac{1}{4}),$ let
$$M_2^{\times}(\psi)=\left\{(x_1,x_2)\in [0,1]^{2}\colon \|qx_1\|\|qx_2\|<\psi(q) \text{ for infinitely many } q\in \N\right\}$$
denote the set of multiplicative $\psi$-well approximable points $(x_1,x_2)\in [0,1]^{2}.$ Assuming the monotonicity of $\psi,$
Gallagher's Theorem asserts that the Lebesgue measure of $M_2^{\times}(\psi)$ is either 0 or 1 according as the series $\sum_q\psi(q)\log\frac{1}{\psi(q)}$ converges or diverges. Without assuming the monotonicity, Beresnevich-Haynes-Velani \cite{BHV13} showed a dichotomy for the Lebesgue measure of $M_2^{\times}(\psi)$ under some additional  assumptions. Removing the additional conditions, Fr\"{u}hwirth-Hauke \cite{FH24} proved the following result$\colon$    for almost all $(x_1,x_2)\in[0,1]^{2},$ there
exist infinitely many $q$ such that $\prod_{i=1}^{2}|qx-p_i|<\psi(q)$ with $p_1, p_2$ both coprime to $q$,
if and only if the series $\sum_q\frac{\phi(q)\psi(q)}{q}\log\big(\frac{q}{\phi(q)\psi(q)})$ diverges.

Hussain-Simmons \cite{HS18} proved that, if $\psi$ tends monotonically to $0$ as $q\to\infty$,
$$\dim_{\rm H}M_2^{\times}(\psi)= 1+\min\{d(\psi), 1\},$$
where $d(\psi)=\inf\big\{s\in[0,1]\colon \sum_{q=1}^{\infty}q\big(\frac{\psi(q)}{q}\big)^{s}<\infty \big\}.$ Combining the product formula of Hausdorff dimension with the Hausdorff measure version of the Duffin-Schaeffer  due to Beresnevich-Velani \cite{BV06},
Fr\"{u}hwirth-Hauke proved that the Hausdorff dimension   remains unchanged even upon removing the monotonicity of the approximation function.
\begin{thm}[Fr\"{u}hwirth-Hauke, \cite{FH24}]\label{FH24-1}
Let $\psi\colon \N\to [0,\frac{1}{4})$ be an arbitrary function. Then
\begin{equation*}
\dim_{\rm H}M_2^{\times}(\psi) = 1+\min\{d(\psi), 1\}.
  \end{equation*}
\end{thm}

In view of the study of the Fourier dimension of   $W(\psi),$
the following question arises naturally: whether is $M_2^ {\times}(\psi)$  (non-)Salem?  or, what is the Fourier dimension of $M_2^{\times}(\psi)$?
Notably, Fourier dimension of $M_2^{\times}(\psi)$ is trickier to deal with than its Hausdorff dimension. In general,  the produce formula
$$\dim_{\rm F}(\mu\times\nu)\ge \dim_{\rm F}\mu+\dim_{\rm F}\nu$$
 is not true 
 since
$$\dim_{\rm F}(\mu\times\nu)=\min\{\dim_{\rm F}\mu, \dim_{\rm F}\nu\},$$
unless $\dim_{\rm F}\mu=\dim_{\rm F}\nu=0.$  See \cite{F24,F24+} for more details about the Fourier decay of product measures.
Moreover, it is difficult to check whether or not   $\mu\in \mathcal{M}(M_2^{\times}(\psi))$ satisfies a desired power-like decay. 

As in the linear case, we introduce the following set: Let $Q$ be an infinite subset of $\mathbb{N},$  and define
$$M_2^{\times}(\psi,Q):=\left\{(x_1,x_2)\in [0,1]^{2}\colon \|qx_1\|\|qx_2\|<\psi(q) \text{ for infinitely many } q\in Q\right\}.$$
Specially, $M_2^{\times}(\psi)=M_2^{\times}(\psi,\mathbb{N}).$ 
We obtain the Fourier dimension of this set.
\begin{thm}\label{TZ24}
 Let $\psi\colon \N\to [0,\frac{1}{4})$ be an arbitrary function satisfying $\sum_{q\in Q}\psi(q)\log\frac{1}{\psi(q)}$ converges. Then
 \begin{equation*}
 \dim_{\rm F}M_2^{\times}(\psi,Q)
       =2\tau(\psi,Q),     \end{equation*}
where $\tau(\psi,Q)=\inf\big\{s\in[0,1]\colon \sum_{q\in Q}q^{-s}(\psi(q))^{\frac{s}{2}}<\infty\big\}.$
\end{thm}

Let us make the following remarks regarding Theorem \ref{TZ24}$\colon$
\begin{itemize}
\item We can extend our result to the inhomogeneous setting$\colon$
$$M_2^{\times}(\psi,Q, {\bf{y}}):=\left\{(x_1,x_2)\in [0,1]^{2}\colon \|qx_1-y_1\|\|qx_2-y_2\|<\psi(q) \text{ for infinitely many } q\in Q\right\}$$
where ${\bf{y}}=(y_1,y_2)\in [0,1]^{2}.$ The proof of Theorem \ref{TZ24} applies to this setting to show that   $\dim_{\rm F}M_2^{\times}(\psi,Q,{\bf{y}})=2\tau(\psi,Q)$.

\item The Fourier dimension formula in Theorem \ref{TZ24} does not hold when $\sum_{q}\psi(q)\log\frac{1}{\psi(q)}$ diverges.  A  counterexample is provided in Example 1.
\end{itemize}

\noindent\textbf{Example 1.} Consider the function
$\psi(q)=\frac{1}{q}$ for $q\in\N$. It is readily checked that {$\sum_{q\in \N}\psi(q)\log\frac{1}{\psi(q)}$ diverges} and $\tau(\psi,\N)=\frac{2}{3}.$  On the other hand, by Dirichlet's Theorem\footnote{For any $x\in \mathbb{R},$ there are infinitely many $q\in \mathbb{N}$ such that $\|qx\|< \frac{1}{q}$.},$$M_2^{\times}(\psi)=[0,1]^{2} .$$
Hence $\dim_{\rm F}M_2^{\times}(\psi)=2\not=2 \tau(\psi,\N).$
\medskip

A  direct corollary of Theorem  \ref{TZ24}  is the following.
\begin{cor}
For $\psi(q)=q^{-\tau}$ with $\tau>1,$ the set $M_2^{\times}(\psi)$ is non-Salem.
\end{cor}

\begin{proof}
By Theorems \ref{FH24-1}, \ref{TZ24}, we deduce that
$$\dim_{\rm H}M_2^{\times}(\psi)=\frac{\tau+3}{\tau+1}, \ \dim_{\rm F}M_2^{\times}(\psi)=\frac{4}{\tau+2}.$$
It is readily check that $\frac{4}{\tau+2}<\frac{\tau+3}{\tau+1}$, and thus $M_2^{\times}(\psi)$ is non-Salem.
\end{proof}

\section{Proof of Theorem \ref{TZ1}}
 We construct a function $\psi$ satisfying
$\sum_{q\in \N}\psi(q)=\infty$ and
$\dim_{\rm F}W(\psi,\N)=0.$

Let  $\mathcal{P}$ denote the set of all prime numbers. The series $\sum_{q\in \mathcal{P}}\frac{1}{p}$ diverges.
Now we define the desired function $\psi$ inductively.

\noindent Level 1$\colon$ Choose prime numbers $p_1^{(1)}<p_2^{(1)}<\ldots<p_{M_1}^{(1)}$ satisfying
\begin{equation}\label{sum1}
 \frac{1}{p_1^{(1)}}+\cdots+\frac{1}{p_{M_1}^{(1)}}>2.
\end{equation}
 Set $\mathcal{P}_1=\{p_1^{(1)},\ldots, p_{M_1}^{(1)}\}$ and let  $N^{(1)}=\prod_{i=1}^{M_1}p_i^{(1)}.$ We define
\begin{equation*}
\psi_1(q)=\left\{
    \begin{array}{ll}
      \frac{q}{2N^{(1)}}, & \ \ \ \text{if }~ q| N^{(1)}; \\
      ~&
       \\
      0, & \ \ \ \text{otherwise.}
    \end{array}
  \right.
\end{equation*}
By (\ref{sum1}), $\sum_{q\in \mathcal{P}_1}\psi_1(\frac{N^{(1)}}{q})>1$, and thus   $\sum_{q}\psi_1(q)>1$. Moreover if there is some $q$ satisfying $\|qx\|<\psi_1(q)$, we have that $\|N^{(1)}x\|<2^{-1}.$

\medskip

\noindent Level $k(\ge2)\colon$ Having defined $\mathcal{P}_{k-1}$, $N^{(k-1)}$ and $\psi_{k-1}$, we choose $\mathcal{P}_{k}$ as follows.

We take $\mathcal{P}_k=\{p_1^{(k)},\ldots, p_{M_k}^{(k)}\}$ consisting of prime numbers with
\begin{align*}
&p_{i}^{(k)}>N^{(k-1)} \text{ for } i=1,\ldots, M_k,\\
&\frac{1}{p_1^{(k)}}+\cdots+\frac{1}{p_{M_k}^{(k)}}>2^{k}.
\end{align*}
Putting  $N^{(k)}=\prod_{i=1}^{M_k}p_i^{(k)}$, we define
\begin{equation*}
\psi_k(q)=\left\{
    \begin{array}{ll}
      \frac{q}{2^{k}N^{(k)}}, & \ \ \ \text{if }~ q| N^{(k)}; \\
      ~&
       \\
      0, & \ \ \ \text{otherwise.}
    \end{array}
  \right.
\end{equation*}
Similarly, it is readily check that $\sum_{q}\psi_k(q)>1$ and
\begin{equation}\label{transf}
  \|qx\|<\psi_k(q) \text{ for some }q \Longrightarrow \|N^{(k)}x\|<2^{-k}.
\end{equation}

\medskip

Now we define the desired function $\psi\colon \N\to \mathbb{R}$ by $$\psi(q)=\sum\limits_{k=1}^{\infty}\psi_{k}(q).$$
Remark that each $\psi_k$ has a finite support, and these supports are pairwise disjoint, therefore the summation above contains at most one non-zero term. And $\psi$ has the following properties$\colon$
\medskip

(i) $\sum_{q\in\N}\psi(q)
\ge \sum_{k}\sum_{q}\psi_{k}(q)=\infty.$
\medskip

(ii) If $x\in W(\psi,\N)$, there are infinitely many $k$ such that
 $\|qx\|<\psi_k(q)$  for some $q$, and thus, by (\ref{transf}), $ \|N^{(k)}x\|<2^{-k}$. Hence
 $$W(\psi,\N)\subset\big\{x\in[0,1)\colon \|N^{(k)}x\|<2^{-k} \text{ for infinitely many } k\in \N\big\}=:E.$$
  Since $\sum_k2^{-k}<\infty$ and $\sum_k(\frac{2^{-k}}{N^{(k)}})^{s}<\infty$ for any $s>0$, we apply
 Theorem \ref{CH24} to obtain  $\dim_{\rm F}E=0$, and thus $\dim_{\rm F}W(\psi,\N)=0$ as desired.

\section{Proof of Theorem \ref{TZ24}}

We divide the proof of Theorem \ref{TZ24} into the following two propositions.

\begin{proposition}\label{up}
If $\sum_{q\in {Q}}\psi(q)\log\frac{1}{\psi(q)}<\infty,$ then $\dim_{\rm F}M_2^{\times}(\psi,{Q})\le 2\tau(\psi,{Q}).$
\end{proposition}

\begin{proposition}\label{low}
$\dim_{\rm F}M_2^{\times}(\psi,{Q})\ge 2\tau(\psi,{Q}).$
\end{proposition}

\subsection{Proof of Proposition \ref{up}}
We'll proceed   by contradiction.  To this end, we assume $\dim_{\rm F}M_2^{\times}(\psi,{Q})=2s>2\tau(\psi,{Q})$, and thus  there exists a Borel probability measure $\mu$  which gives full measure to $M_2^{\times}(\psi,{Q})$ and whose Fourier transform satisfies that
$$|\widehat{\mu}(\xi)|\ll|\xi|^{-s}\quad \text{ for }|\xi|\ge1.$$
Since $s>\tau(\psi,{Q})$, for $0<\varepsilon<s-\tau(\psi,{Q}),$ we have that
$$\sum_{q\in {Q}}q^{-(s-\varepsilon)}(\psi(q))^{\frac{s-\varepsilon}{2}}<\infty.$$

We  will reach  a contradiction by showing that $\mu(M_2^{\times}(\psi,{Q}))=0$, which is achieved by using the limit-superior structure of $M_2^{\times}(\psi,{Q})$ and applying the first Borel-Cantelli lemma.

We start with the limit-superior structure of $M_2^{\times}(\psi,{Q})$:
\begin{align*}
M_2^{\times}(\psi,{Q})&=\big\{(x_1,x_2)\in[0,1]^{2}\colon \|qx_1\|\|qx_2\|<\psi(q) \text{ for infinitely many } q\in {Q}\big\}\\
                                                        &= \bigcap_{N=1}^{\infty}\bigcup_{ {N\le q\in Q}}A_q,
\end{align*}
where $A_q:=\big\{(x_1,x_2)\in [0,1]^{2}\colon \|qx_1\|\|qx_2\|<\psi(q) \big\}$ consists of  $q^2$  ``star-shaped'' domains with   centers at the rational points $(\frac{a}{q},\frac{b}{q})$.

\medskip

%

We next estimate the $\mu$-measure of $A_q$. To do this, we cover $A_q$ by rectangles and use  Fourier analysis. Define
$$R_{q,j}:=\left\{(x_1,x_2)\in[0,1]^{2}\colon \|qx_1\|<q\cdot 2^{-(j-1)}, \|qx_2\|<\frac{\psi(q)}{q\cdot 2^{-j}}\right\},$$
and $\mathcal{I}_q:=\{j\in\N\colon  2q\le 2^j\le \frac{q}{\psi(q)}\}.$
We claim that
\begin{equation}\label{inclu}
  A_q\subset \bigcup_{j\in \mathcal{I}_q}R_{q,j}.
\end{equation}
 In fact, if $(x_1,x_2)\in A_q$, then $q\cdot 2^{-j}\le \|qx_1\|<q\cdot 2^{-(j-1)}$ for some $j\in\N$, and thus
 $$\|qx_2\|<\frac{\psi(q)}{\|qx_1\|}\le \frac{\psi(q)}{q\cdot 2^{-j}}.$$
 We obtain that $x\in R_{q,j}$.
 Noting that $\|x\|\le\frac12$, we deduce that such $j$ satisfies
 $q\cdot 2^{-j}\le\frac12$, and thus $j\ge j_0=\lfloor\log_2 2q\rfloor$. On the other hand, we readily check that $\frac{\psi(q)}{q\cdot 2^{-j}}\ge 1$ if $j\ge j_1=\lfloor\log_2 \frac q{\psi(q)}\rfloor$, and thus $\|qx_2\|<\frac{\psi(q)}{q\cdot 2^{-j}}$ holds trivially. Hence
 $R_{q,j}\subset R_{q,j_1}$ for $j\ge  j_1$.

 \medskip

 We evaluate the Lebesgue measure of $A_q$ by further decomposing it into rectangles
 $$R_{q,j}(a,b)=\left\{(x_1,x_2)\in[0,1]^{2}\colon \left|x_1-\frac{a}{q}\right|<2^{-(j-1)},
\left|x_2-\frac{b}{q}\right|< \frac{\psi(q)}{q^{2}\cdot 2^{-j}}\right\},$$
where $(a,b)\in \{0,1,\ldots,q-1\}^2$.
Since $\mathcal{L}\big(R_{q,{j}}(a,b)\big)\asymp \frac{\psi(q)}{q^{2}}$, we have
$$\mathcal{L}\left(R_{q,{j}}\right)\asymp \psi(q).$$

We write  $\mathcal{X}_{R_{q,j}}$ as the indicator function of $R_{q,j},$
and extend it as a periodic function with respect to the lattice $\mathbb{Z}^{2}.$
As is customary, we write $e(x)=\exp(2\pi ix), $ and for ${\bf{x}}=(x_1,x_2), {\bf{n}}=(n_1,n_2),$ we write ${\bf{n}}\cdot {\bf{x}}=n_1x_1+n_2x_2$. Then $\mathcal{X}_{R_{q,j}}$
has the Fourier series
$$\mathcal{X}_{R_{q,j}}({\bf{x}})=\sum_{{\bf{n}}\in \mathbb{Z}^{2}}c_{q,j}({\bf{n}}) e({\bf{n}}\cdot {\bf{x}}),$$
where
$$c_{q,j}({\bf{n}})=\iint_{[0,1]^{2}}\mathcal{X}_{R_{q,j}}(\x)e(-\n\cdot\x)\d \x.$$
Hence we have
$$\mu(R_{q,j})=\iint_{[0,1]^{2}} \mathcal{X}_{R_{q,j}}({\bf{x}}) \d\mu(\x)=\sum_{{\bf{n}}\in \mathbb{Z}^{2}}c_{q,j}({\bf{n}}) \iint_{[0,1]^{2}}e(\nx)\d\mu
=\sum_{{\bf{n}}\in \mathbb{Z}^{2}}c_{q,j}({\bf{n}})\widehat{\mu}({-\bf{n}}).$$
It follows from (\ref{inclu}) that
$$\mu(A_q)\le \sum_{j\in\mathcal{I}_q}\mu(R_{q,j})=\sum_{j\in\mathcal{I}_q}\sum_{{\bf{n}}\in \mathbb{Z}^{2}}c_{q,j}({\bf{n}})\widehat{\mu}({-\bf{n}}).$$

\medskip
Before   further estimating, we make some remarks on the Fourier coefficient $c_{q,j}(\n)$. Since  $\mathcal{X}_{R_{q,j}}$ is indeed a periodic function with respect to the lattice $\frac1q\cdot\mathbb{Z}^{2},$ $c_{q,j}(\n)$ vanishes if
 $q\!\nmid\! \n$  (that is, either $q\!\nmid\!  n_1$ or $q\!\nmid\!  n_2$).  On the other hand,  when $q|{\bf{n}},$  the periodicity yields that $$c_{q,j}({\bf{n}})=q^{2}\iint_{R_{q,j}(1,1)}e(-\nx)\d \x=q^{2}\int_{-2^{-(j-1)}}^{2^{-(j-1)}}e(-n_1x_1)\d x_1 \int_{-\frac{\psi(q)}{q^{2}2^{-j}}}^{\frac{\psi(q)}{q^{2}2^{-j}}}e(-n_2x_2)\d x_2.$$
 Using the trivial inequality
 $$\int_{-\eta}^{\eta}e(nx)\ll \min\big\{\frac1{|n|},\eta\big\}\quad
\text{(with   $\min\big\{\frac10,\eta \big\}=\eta$ by convention),}$$
we deduce that
$$c_{q,j}({\bf{n}})\ll q^{2}\min\big\{\frac{1}{|n_1|}, 2^{-j}\big\}\min\big\{\frac{1}{|n_2|},\frac{\psi(q)}{q^{2}2^{-j}}\big\}.$$
Recalling that  $|\widehat{\mu}(\xi)|\ll|\xi|^{-s}$, we deduce that
\begin{align*}
\mu(A_q)\le&\sum_{j\in\mathcal{I}_q}\sum_{{\bf{n}}\in \mathbb{Z}^{2}}c_{q,j}({\bf{n}})\widehat{\mu}({-\bf{n}})= \sum_{j\in\mathcal{I}_q}\sum_{{\bf{k}}\in \mathbb{Z}^{2}}c_{q,j}({q\bf{k}})\widehat{\mu}({-q\bf{k}})   \\
  \ll & \sum_{j\in\mathcal{I}_q}\sum_{{\bf{k}}\in \mathbb{Z}^{2}} \min\big\{\frac{1}{|k_1|}, q  2^{-j}\big\}\min\big\{\frac{1}{|k_2|},\frac{ \psi(q)}{q 2^{-j}}\big\}\cdot \big(q\max\{|k_1|,|k_2|\}\big)^{-s}\\
 \ll   & \sum_{j\in\mathcal{I}_q}\sum_{{\bf{k}}\in \mathbb{\N}^{2}} \min\big\{\frac{1}{k_1}, q  2^{-j}\big\}\min\big\{\frac{1}{k_2},\frac{ \psi(q)}{q 2^{-j}}\big\}\cdot q^{-s}\min\big\{k_1^{-s},k_2^{-s}\big\}\\
 =& \sum_{j\in\mathcal{I}_q}\sum_{{\bf{k}}\in \mathbb{\N}^{2}} S(j,\k),
\end{align*}
where  $S(j,\k)=
\min\big\{\frac{1}{k_1}, q  2^{-j}\big\}\min\big\{\frac{1}{k_2},\frac{ \psi(q)}{q 2^{-j}}\big\}\cdot q^{-s}\min\big\{k_1^{-s},k_2^{-s}\big\}.$\footnote{When $\k=0$, $\widehat{\mu}(q\k)=1$, and $S(j,{\bf{0}})=\psi(q)$.}


%
%
 \begin{lem}\label{muAq}
For
 $\varepsilon>0$ we have $$ \mu(A_q)\ll \psi(q)\log\frac{1}{\psi(q)}+q^{-(s-\varepsilon)}(\psi(q))^{\frac{s-\varepsilon}{2}}.
$$
 \end{lem}
\begin{proof}
 We need to bound the summation $\sum_{j\in\mathcal{I}_q}\sum_{{\bf{k}}\in \mathbb{\N}^{2}} S(j,\k)$.  In the following proof, we abbreviate $\mathcal{I}_q$ to $\mathcal{I}$. We remark that within the proof,  all the implied constants  in   Vinogradov's notation are independent of $q$ (while they may depend on $s$).

We first partition $\N^2$ into four subclasses, and deal with the summations over these subclasses separately.
\medskip

\noindent\underline{Case 1.}\quad $\Omega_0=\{\textbf{0}\}$.
 $$\sum_{j\in\mathcal{I}}S(j,\textbf{0})=\sum_{j\in\mathcal{I}}\psi(q)\asymp \psi(q)\log\frac{1}{\psi(q)}.$$

\medskip

\noindent\underline{Case 2.}\quad  $\Omega_1=\{\k=(k_1,k_2) \in\mathbb{N}^{2}\colon k_1=0, k_2\ge 1\}.$

\smallskip

We obtain
\begin{align*}
 \sum_{{\bf{k}}\in \Omega_1}S(j,\k)
 &\ll \sum_{k\ge 1} q 2^{-j}  \min\big\{\frac{1}{k}, \frac{\psi(q)}{q 2^{-j}}\big\} \cdot(qk)^{-s} \\
 & =\sum_{k\le \frac{q2^{-j}}{\psi(q)}} q 2^{-j}   \frac{\psi(q)}{q 2^{-j}} \cdot(qk)^{-s} +\sum_{k>\frac{q2^{-j}}{\psi(q)}}q 2^{-j}   \frac{1}{k} \cdot(qk)^{-s}\\
 &\ll q^{1-2s}(\psi(q))^{s}2^{-j(1-s)},
\end{align*}
where in the last inequality we use   the facts $\sum_{1\le k\le \xi}k^{-s}\ll \xi^{1-s}$ and $\sum_{k> \xi}k^{-(1+s)}\ll \xi^{-s}.$
Hence we have
$$ \sum_{j\in\mathcal{I}} \sum_{{\bf{k}}\in \Omega_1}S(j,\k)
\ll \sum_{j\ge \log 2q}q^{1-2s}(\psi(q))^{s}2^{-j(1-s)}\ll q^{-s}(\psi(q))^{s}.$$

\medskip

\noindent\underline{Case 3.}\quad  $\Omega_2=\{\k=(k_1,k_2) \in\mathbb{N}^{2}\colon k_1\ge1, k_2=0\}.$

  \smallskip

 Similarly to the case (1), we obtain
\begin{align*}
 \sum_{{\bf{k}}\in \Omega_2}S(j,\k)
 &
=\sum_{k\ge1}\min\big\{\frac{1}{k}, q  2^{-j}\big\} \frac{ \psi(q)}{q 2^{-j}} \cdot (qk)^{-s}
 \\& =\sum_{k\le\frac{2^{j}}{q}}  q  2^{-j} \frac{ \psi(q)}{q 2^{-j}} \cdot (qk)^{-s}
  +\sum_{k>\frac{2^{j}}{q}}   \frac{1}{k} \frac{ \psi(q)}{q 2^{-j}} \cdot (qk)^{-s}   \\
 &\ll q^{-1}\psi(q)2^{j(1-s)},
\end{align*}
and thus
$$ \sum_{j\in\mathcal{I}} \sum_{{\bf{k}}\in \Omega_2}S(j,\k)
\ll \sum_{1\le j\le \log\frac{2q}{\psi(q)}}q^{-1}(\psi(q))2^{j(1-s)}\ll q^{-s}(\psi(q))^{s}.$$

\medskip

\noindent\underline{Case 4.}\quad  $\Omega_3=\{\k=(k_1,k_2) \in\mathbb{N}^{2}\colon k_1\ge1, k_2\ge1\}.$

 \smallskip

 In this case, we divide   the summation $T:=\sum_{j\in\mathcal{I}} \sum_{{\bf{k}}\in \Omega_3}S(j,\k) $ into several parts by partition the domain of summation  $\mathcal{I}\times \Omega_3$.

We first divide   $\mathcal{I}=\{i\in\N\colon \log 2q\le i\le \log\frac{q}{\psi(q)}\}$  into two parts
$$
\mathcal{I}^1=\big\{i\in\N\colon ~\log 2q\le i< \log\frac{q}{\sqrt{\psi(q)}}\big\},\quad
\mathcal{I}^2=\big\{i\in\N\colon  \log\frac{q}{\sqrt{\psi(q)}}\le i\le \log\frac{q}{\psi(q)}\big\};
$$
 divide $\Omega_3$ into two parts
 $$\Omega_3^1=\big\{\k\in \Omega_3 \colon k_1\le k_2 \big\},\quad \Omega_3^2=\big\{\k\in \Omega_3 \colon k_1> k_2 \big\}.$$
 Remark that $ q  2^{-j}>\frac{ \psi(q)}{q 2^{-j}}$ if $j\in \mathcal{I}^1$, while  $ q  2^{-j}\le \frac{ \psi(q)}{q 2^{-j}}$ if $j\in \mathcal{I}^2$.
 In this way, we divide the summation $T$ into four parts $T^{uv}$  with $u,v\in\{1,2\}$, where  $T^{uv}$ is the summation of $S(j,\k)$ with $(j,\k)$ runing over  $\mathcal{I}^u\times \Omega_3^v$.

 \medskip

 Now we   estimate $T^{uv}$.

 \medskip

\noindent \underline{\text{Estimation of $T^{11}$}}.~ We have
$$T^{11}=\sum_{j\in\mathcal{I}^1} \sum_{k_2=1}^{\infty}\sum_{k_1=1}^{k_2}q^{-s}k_2^{-s}\min\big\{\frac{1}{k_1}, q2^{-j}\big\}\min\big\{\frac{1}{k_2},\frac{\psi(q)}{q2^{-j}}\big\}. $$

In the following, we use the following basic estimation$\colon$
$$ \min\big\{\frac{1}{k_1}, q2^{-j}\big\}=\begin{cases}
     q2^{-j} & \text{if $k_1\le \frac{2^{j}}{q}$}, \\
    ~~ \frac{1}{k_1} & \text{otherwise},
\end{cases}\quad
\min\big\{\frac{1}{k_2},\frac{\psi(q)}{q2^{-j}}\big\}=\begin{cases}
     \frac{\psi(q)}{q2^{-j}}& \text{if $k_2\le \frac{q2^{-j}}{\psi(q)}$},\\
     ~~\frac{1}{k_2} & \text{otherwise}.
     \end{cases}
     $$

  In order to remove the min-symbols from the inner summation
 $$t_j=\sum_{k_2=1}^{\infty}\sum_{k_1=1}^{k_2}q^{-s}k_2^{-s}\min\big\{\frac{1}{k_1}, q2^{-j}\big\}\min\big\{\frac{1}{k_2},\frac{\psi(q)}{q2^{-j}}\big\},$$
we further divide it into several parts$\colon$
\begin{align*}
t^1_j&:=\sum_{1\le  k_2 \le \frac{2^{j}}{q}}\sum_{k_1=1}^{k_2}q^{-s} k_2 ^{-s}\cdot q2^{-j}\cdot \frac{\psi(q)}{q2^{-j}},\\
t_j^{2}&:=\sum_{\frac{2^{j}}{q}<  k_2 \le \frac{q2^{-j}}{\psi(q)}}\sum_{k_1=1}^{k_2}q^{-s} k_2 ^{-s}\min\big\{\frac{1}{ k_1 }, q2^{-j}\big\}\cdot \frac{\psi(q)}{q2^{-j}},\\
t_j^{3}&:=\sum_{  k_2 > \frac{q2^{-j}}{\psi(q)}}\sum_{k_1=1}^{k_2}q^{-s} k_2 ^{-s}\min\big\{\frac{1}{ k_1 }, q2^{-j}\big\}\cdot \frac{1}{ k_2}.
\end{align*}
For $t^1_j$, we have
$$t^1_j=q^{-s}\psi(q)\sum_{1\le  k_2 \le \frac{2^{j}}{q}} k_2 ^{1-s}\asymp q^{-2}\psi(q)2^{j(2-s)}.$$
For $t^2_j$, we deduce
\begin{align*}
t^2_j&\le \sum_{\frac{2^{j}}{q}<  k_2 \le \frac{q2^{-j}}{\psi(q)}}\Big(\sum_{1\le  k_1 \le \frac{2^{j}}{q}}q^{-s}\psi(q) k_2 ^{-s}+\sum_{\frac{2^{j}}{q}\le  k_1 \le k_2}q^{-1-s}\psi(q)2^{j} k_2 ^{-s}\cdot { k_1 }^{-1}\Big)\\
&\ll \sum_{\frac{2^{j}}{q}<  k_2 \le \frac{q2^{-j}}{\psi(q)}}q^{-1-s}\psi(q)2^{j} k_2 ^{-s}\log  k_2
\ll 2^{js}q^{-2s}(\psi(q))^{s}\log\frac{1}{\psi(q)}.
\end{align*}
For $t^3_j$, we obtain
\begin{align*}
t^3_j&\le \sum_{  k_2 > \frac{q2^{-j}}{\psi(q)}}\Big(\sum_{1\le  k_1 \le \frac{2^{j}}{q}}q^{1-s} k_2 ^{-1-s}2^{-j}+\sum_{\frac{2^{j}}{q}\le  k_1 \le k_2}q^{-s} k_2 ^{-1-s}\cdot { k_1 }^{-1}\Big)\\
&\ll \sum_{  k_2 > \frac{q2^{-j}}{\psi(q)}}q^{-s} k_2 ^{-1-s}\log k_2
\ll 2^{js}q^{-2s}(\psi(q))^{s}\log\frac{1}{\psi(q)} .
\end{align*}

Substituting these into $t_j=t_j^{1}+t_j^{2}+t_j^{3}$ yields that
$$\sum_{j\in \mathcal{I}^1}t_j \ll \sum_{j\in \mathcal{I}^1}\big(q^{-2}\psi(q)2^{j(2-s)}+2^{js}q^{-2s}(\psi(q))^{s}\log\frac{1}{\psi(q)}\big)\ll q^{-(s-\varepsilon)}(\psi(q))^{\frac{s-\varepsilon}{2}},$$
where the last inequality is due to the fact that $(\frac{1}{\psi(q)})^{\varepsilon}>\log \frac{1}{\psi(q)}$ as $\frac{1}{\psi(q)}\to \infty.$

\medskip
\noindent \underline{\text{Estimation of $T^{12}$}}.
$$T^{12}=\sum_{j\in\mathcal{I}^1} \sum_{k_2=1}^{\infty}\sum_{k_1>k_2}q^{-s}k_2^{-s}\min\big\{\frac{1}{k_1}, q2^{-j}\big\}\min\big\{\frac{1}{k_2},\frac{\psi(q)}{q2^{-j}}\big\}. $$
Similarly,  we divide the inner summation $t_j$ into three parts and estimate as follows$\colon$
\begin{align*}
t_j^{1}&:=\sum_{1\le  k_2 \le \frac{2^{j}}{q}}\sum_{  k_1 \ge  k_2 }q^{-s} k_1 ^{-s}\cdot\min\big\{\frac{1}{ k_1}, q2^{-j}\big\}\cdot \frac{\psi(q)}{q2^{-j}}\\
&\le \sum_{1\le  k_2 \le \frac{2^{j}}{q}}\Big(\sum_{ k_2 \le  k_1 <\frac{2^{j}}{q}}q^{-s}\psi(q) k_1 ^{-s}+\sum_{ k_1 \ge \frac{2^{j}}{q}}q^{-1-s}2^{j}\psi(q)k_1^{-1-s}\Big)\\
&\ll \sum_{1\le  k_2 \le \frac{2^{j}}{q}}q^{-1}\psi(q)2^{j(1-s)}\le q^{-2}\psi(q)2^{j(2-s)};\\
t_j^{2}&:=\sum_{\frac{2^{j}}{q}<  k_2 \le \frac{q2^{-j}}{\psi(q)}}\sum_{ k_1 \ge  k_2 }q^{-s} k_1 ^{-s}\cdot \frac{1}{ k_1 }\cdot \frac{\psi(q)}{q2^{-j}}\\
&\ll \sum_{\frac{2^{j}}{q}<  k_2 \le \frac{q2^{-j}}{\psi(q)}}q^{-1-s}2^{j}\psi(q)|k_{2}|^{-s}\ll q^{-2s}(\psi(q))^{s}2^{js};\\
t_j^{3}&:=\sum_{  k_2 > \frac{q2^{-j}}{\psi(q)}}\sum_{ k_1 \ge  k_2 }q^{-s} k_1 ^{-s}\cdot \frac{1}{ k_1}\cdot \frac{1}{ k_2} \ll \sum_{  k_2 > \frac{q2^{-j}}{\psi(q)}}q^{-s} k_2 ^{-1-s}\asymp q^{-2s}(\psi(q))^{s}2^{js}.
\end{align*}

Combining these yields that
$$T^{12}=\sum_{j\in \mathcal{I}^{1}}t_j\ll \sum_{j\in \mathcal{I}^{1}}\big(q^{-2}\psi(q)2^{j(2-s)}+q^{-2s}(\psi(q))^{s}2^{js}\big)\ll q^{-s}(\psi(q))^{\frac{s}{2}}.$$


\medskip
\noindent \underline{\text{Estimation of $T^{21}+T^{22}$}}.~
 Similar arguments apply to these case:
\begin{align*}
T^{21}+T^{22}&\ll \sum_{j\in\mathcal{I}^2}\Big(\sum_{k_2=1}^{\infty}\sum_{ k_1=1}^{ k_2 }q^{-s} k_2 ^{-s}\min\big\{\frac{1}{ k_1 }, q2^{-j}\big\}\min\big\{\frac{1}{ k_2 },\frac{\psi(q)}{q2^{-j}}\big\}\\
& ~~~~~~~+ \sum_{k_2=1}^{\infty}\sum_{ k_1 \ge  k_2 }q^{-s} k_1 ^{-s}\min\big\{\frac{1}{ k_1}, q2^{-j}\big\}\min\big\{\frac{1}{ k_2},\frac{\psi(q)}{q2^{-j}}\big\}\Big)\\
& \ll q^{-(s-\varepsilon)}(\psi(q))^{\frac{s-\varepsilon}{2}}.
\end{align*}

\smallskip
\noindent To sum up,  we have
$$\sum_{j\in\mathcal{I}} \sum_{{\bf{k}}\in \Omega_3}S(j,\k) =T=\sum_{u,v}T^{uv} \ll  q^{-(s-\varepsilon)}(\psi(q))^{\frac{s-\varepsilon}{2}}.$$

\bigskip

Combining four cases, we obtain that
$$\sum_{j\in\mathcal{I}_q}\sum_{{\bf{k}}\in \mathbb{\N}^{2}} S(j,\k) \ll \psi(q)\log\frac{1}{\psi(q)}+q^{-(s-\varepsilon)}(\psi(q))^{\frac{s-\varepsilon}{2}},$$
which completes the proof of the lemma.
\end{proof}

Finally, by Lemma \ref{muAq} we have   that
$$\sum_{q\in {Q}}\mu(A_q)\le \sum_{q\in {Q}}\left(\psi(q)\log\frac{1}{\psi(q)}+q^{-(s-\varepsilon)}(\psi(q))^{\frac{s-\varepsilon}{2}}\right)<\infty.$$
We deduce by  the first Borel-Cantelli lemma that $\mu(M_{2}^{\times}(\psi,{Q}))=0$, the desired contradiction.

\subsection{Proof of Proposition \ref{low}}
Before proceeding, we cite a Fourier dimension result of the set
$${S(\Psi,Q):=\left\{(x_1,x_2)\in[0,1]^{2}\colon \|qx_1\|<\Psi(q), \|qx_2\|<\Psi(q) \text{ for infinitely many }q\in Q\right\}}.$$
\begin{lem}[\cite{CH24}, Proposition 1.4.4]
{Let $Q$ be an infinite subset of $\N$}.
Let $\Psi\colon \N\to [0,\frac{1}{2})$ be a positive function. Then
$$\dim_{\rm F}S(\Psi,{Q})\ge 2\lambda(\Psi,{Q}),$$
where $\lambda(\Psi,{Q})=\inf\left\{s\in[0,1]\colon \sum_{q\in {Q}}\big(\frac{\Psi(q)}{q}\big)^{s}<\infty\right\}.$
\end{lem}

\bigskip
Putting  $\Psi(q)=(\psi(q))^{\frac{1}{2}}$, we readily check that
$\lambda(\Psi,Q)=\tau(\psi,Q)$. Moreover, we have $S(\Psi,{Q})\subset M_{2}^{\times}(\psi,Q)$, and  thus
$$\dim_{\rm F}M_{2}^{\times}(\psi,Q)\ge 2\tau(\psi,Q).$$

\subsection*{Acknowledgements}
This work was supported by NSFC Nos. 12171172, 12201476.

\author{Bo Tan}
{\footnotesize

School  of  Mathematics  and  Statistics

Huazhong  University  of Science  and  Technology, 430074 Wuhan, PR China

Email: \texttt{tanbo@hust.edu.cn}}
\vspace{5mm}

\author{Qing-Long Zhou}
{\footnotesize

School  of  Mathematics  and  Statistics

 Wuhan University of Technology, 430070 Wuhan, PR China

Email: \texttt{zhouql@whut.edu.cn}}

\end{document}